\documentclass[11pt]{amsart}
\usepackage{amssymb, amscd, amsmath, float, graphicx, longtable,color,
  enumerate, fullpage
}
\usepackage[textwidth=2cm,textsize=small]{todonotes}
\usepackage{amsrefs} 
\usepackage{mathrsfs}
\usepackage{psfrag}
\usepackage{fouriernc}\DeclareMathAlphabet{\mathcal}{OMS}{cmsy}{m}{n}
\usepackage[all]{xy}
\SelectTips{eu}{} 
\xyoption{curve}

\usepackage[utf8]{inputenc}    
\usepackage[T1]{fontenc}       

\newcommand{\hide}[1]{}

\selectfont
\usepackage[colorlinks]{hyperref}

\newcommand\blfootnote[1]{%
  \begingroup
  \renewcommand\thefootnote{}\footnote{#1}%
  \addtocounter{footnote}{-1}%
  \endgroup
}

\DeclareMathOperator{\add}{add{}}

\DeclareMathOperator{\depth}{depth}

\DeclareMathOperator{\End}{End}
\DeclareMathOperator{\Ext}{Ext}

\DeclareMathOperator{\gldim}{gldim}

\DeclareMathOperator{\Hom}{Hom}

\DeclareMathOperator{\rank}{rank}

\DeclareMathOperator{\Spec}{Spec}

\DeclareMathOperator{\syz}{syz{}}

\renewcommand{\phi}{\varphi}

\renewcommand{\leq}{\leqslant}

\renewcommand{\geq}{\geqslant}

\newcommand{\nccr}{\emph{noncommutative crepant resolution}}

\renewcommand{\to}{\longrightarrow}

\newcommand{\p}{{\mathfrak{p}}}
\newcommand{\m}{{\mathfrak{m}}}

\newcommand{\ZZ}{{\mathbb Z}}

\theoremstyle{plain}
\newtheorem{theorem}{Theorem}

\newtheorem{prop}[theorem]{Proposition}
\newtheorem{proposition}[theorem]{Proposition}
\newtheorem{lem}[theorem]{Lemma}
\newtheorem{lemma}[theorem]{Lemma}

\newtheorem{cor}[theorem]{Corollary}

\newtheorem{question}[theorem]{Question}

\newtheorem*{conjecture*}{Conjecture}
\newtheorem*{theorem*}{Theorem}
\newtheorem*{prop*}{Proposition}

\theoremstyle{definition}

\newtheorem{definition}[theorem]{Definition}

\newtheorem{remark}[theorem]{Remark}

\newtheorem{example}[theorem]{Example}

\newtheorem{Hypo}[theorem]{Hypotheses}
\numberwithin{theorem}{section}
\numberwithin{equation}{section}

\begin{document}

\title[]{%
Gorenstein and Totally Reflexive Orders
}

 \author[J.J. Stangle]{Josh Stangle}

\thanks{}

\date{\today}



\begin{abstract} In this paper we study orders over Cohen-Macaulay rings. We discuss the properties needed for these orders to give noncommutative crepant resolutions of the base rings; namely, we want algebraic analogs of birationality, nonsingularity, and crepancy. While some definitions have been made, we discuss an alternate definition and obstructions to the existence of such objects. We then give necessary and sufficient conditions for an order to have certain desirable homological properties. We examine examples of rings satisfying these properties to prove that certain endomorphism rings over abelian quotient singularities have infinite global dimension.

\end{abstract}
\maketitle

\section{Introduction}

\blfootnote{\textit{Keywords}. Noncommutative crepant resolutions, Gorenstein orders, Endormorphism Rings}

In \cite{VDB1}, Van den Bergh introduced the notion of a \emph{noncommutative (crepant) resolution} of a commutative ring $R$ as an algebraic analog of desingularizations in algebraic geometry. This is a certain endomorphism ring, $\End_R(M)$, over a commutative ring $R$, which is also a maximal Cohen-Macaulay $R$-module and has finite global dimension. A great deal of work has been done in the time since. In the case where $R$ is a Gorenstein ring, Van den Bergh's definition leads to strong theorems and connections to geometric notions. For example, appropriate conditions on $R$ guarantee that $R$ having a non-commutative crepant resolution is equivalent to $\Spec R$ having a commutative crepant resolution. This is the material of Section 2, where some work of Van den Bergh and Stafford-Van den Bergh \cites{VDB1, VDBSTAFF} and Leuschke and Buchweitz-Leuschke-Van den Bergh \cites{Leuschke2, BLV1} is presented. 

In the case where $R$ is not Gorenstein, the situation is much less clear. One goal of study is trying to find a construction in this case which will give some of the main results from the Gorenstein case, in particular Theorems \ref{Gor_equiv}, \ref{vdb1}, and \ref{vdb2}. Work in this setting has been done in large part by Dao, Ingalls, and Faber in \cite{DIF}; Dao, Iyama, Takahashi, and Vial in \cite{DITV}; and Iyama and Wemyss in \cite{IW1}. In these articles, the definition of a noncommutative crepant resolution in the Gorenstein case is applied to the non-Gorenstein case. This leads to some positive results, but the loss of some strong theorems, as in Example \ref{2veron}.

The goal of this note is to discuss an alternative definition of a noncommutative resolution; namely replacing the conditions of $\End_R(M)$ being maximal Cohen-Macaulay with the stronger condition of total reflexivity. It turns out, these cannot exist over non-Gorenstein rings, which is the main result of section 3. In section 4 we discuss the Gorenstein property of an order (which is guaranteed for a noncommutative crepant resolution over Gorenstein $R$) and give necessary and sufficient conditions for $\End_R(R\oplus\omega)$ to be a Gorenstein order when $R$ is Cohen-Macaulay with canonical module $\omega$. Finally in section 5 we use some results of Iyama and Nakajima \cite{Iyama_steady} to find the global dimension of $\End_R(R \oplus \omega)$ in the case of abelian quotient singularities. We begin with a brief review of the case when $R$ is Gorenstein. This paper is a part of the author's dissertation, which is still in progress.  In this paper, all modules are assumed to be finitely generated unless otherwise stated.

\section{The Gorenstein Case}

In this section we treat the case where $R$ is a Gorenstein local normal domain of dimension $d$. Let us recall some definitions.

\begin{definition}
\mbox{}
\begin{itemize} 
\item A finitely generated module $M$ is \emph{maximal Cohen-Macaulay} (MCM) if $\depth_R(M)=\dim(R)$. 

\item An algebra $\Lambda$ is an $R$-\emph{order} if it is an MCM $R$-module.

\item An R-algebra $\Lambda$ is \emph{birational} if $\Lambda \otimes_R K\cong M_n(K)$ where $K$ is the fraction field of $R$.

\item An $R$-algebra $\Lambda$ is \emph{symmetric} if $\Hom_R(\Lambda,R)\cong\Lambda$ as an $\Lambda$-$\Lambda$-bimodules. 

\item $\Lambda$ is \emph{non-singular} if $\gldim(\Lambda_\p)=\dim R_\p$ for all $\p \in \Spec R$.

\end{itemize}
\end{definition}

First, we note that under mild assumptions, non-singularity can be checked only at maximal ideals:

\begin{proposition} \cite{IW1}*{Proposition 2.17} For an order $\Lambda$ over a Cohen-Macaulay ring $R$ with canonical module $\omega$ the following are equivalent: 
\begin{itemize}
\item $\Lambda$ is non-singular.
\item $\gldim \Lambda_\m =\dim R_\m$ for all maximal prime ideals $\m\in \Spec R$. 
\end{itemize}
\end{proposition}

In \cite{Leuschke2}, Leuschke addresses the properties we would like from a noncommutative resolution of singularites. We'd like it to be symmetric, birational and non-singular. In the Gorenstein case these conditions give a concrete description of these orders: 

\begin{theorem}\cite{Leuschke2}*{Theorem 2} Let $R$ be a Gorenstein normal domain of dimension $d$ and $\Lambda$ a module finite $R$-algebra. The following are equivalent:

\begin{itemize}
\item[(i)] $\Lambda$ is a symmetric birational $R$-order and has finite global dimension.
\item[(ii)] $\Lambda \cong \End_R(M)$ for some reflexive $R$-module $M$, and $\Lambda$ is homologically homogenous (all simple $\Lambda$-modules have projective dimension $d$).
\item [(iii)] $\Lambda\cong \End_R(M)$ for some reflexive $R$-module $M$, $\Lambda$ is MCM over $R$ and $\gldim(\Lambda)<\infty$.

\end{itemize}
\label{Gor_equiv}
\end{theorem}

Motivated by this theorem we have the following definition. 

\begin{definition} \label{DEF1}
A \nccr \ of a $d$-dimensional Gorenstein normal domain $R$ is a ring $\Lambda=\End_R(M)$ for a reflexive module $M$ such that $\Lambda$ is an $R$-order and has finite global dimension. \end{definition}

The following theorems of Van den Bergh and Stafford-Van den Bergh show how this definition mirrors the geometric case and that it influences the singularities of $R$.  

\begin{theorem} \cite{VDB1}*{Theorem 6.6.3} Let $R$ be a Gorenstein normal domain which is a finitely generated $k$-algebra with $k$ an algebraically closed field. Assume $R$ is three-dimensional and has terminal singularities. Then $R$ has a noncommutative crepant resolution if and only if $\Spec R$ has a commutative crepant resolution. \label{vdb1}
\end{theorem}

\begin{remark} Note that Theorem \ref{vdb1} is not true in higher Krull dimension, see \cite{DaoExample}. \end{remark}

\begin{theorem}\cite{VDBSTAFF}*{Theorem 1.1} 
Let $\Delta$ be a homologically homogeneous $k$-algebra with $k$ algebraically closed, then $Z=Z(\Delta)$ has at most rational singularities.

In particular, if a normal affine $k$-domain $R$ has a noncommutative crepant resolution then it has rational singularities.  \label{vdb2}
\end{theorem}

\begin{remark} Note in \cite{VDBSTAFF} that the definition of a noncommutative crepant resolution is any homologically homogeneous ring of the form $\Delta=\End_R(M)$ for $M$ reflexive and finitely generated. In the situation where $R$ is not Gorenstein, this is stronger than Definition \ref{DEF1}, which is why this assumption is not needed in Theorem \ref{vdb2}. In the case where $R$ is Gorenstein, any $\Lambda$ satisfying Definition \ref{DEF1} is homologically homogenous and so the theorem remains true with our definition. 
\end{remark}

The work in the Gorenstein case has largely been in producing these resolutions, and there are many results in this direction in \cites{VDB1,BLV1}.

\section{The non-Gorenstein Case}

We now assume only that $R$ is a Cohen-Macaulay normal domain with canonical module $\omega_R$. Following  \cites{DIF,IW1,DITV}, we define

\begin{definition} A \emph{noncommutative crepant resolution} of a Cohen-Macaulay normal domain, $R$,  is $\Lambda=\End_R(M)$ for a reflexive $R$-module $M$ such that $\Lambda$ is a non-singular $R$-order.\end{definition} 

In this case, however, we know Theorem \ref{Gor_equiv} fails \cite{Leuschke2}*{Example P.3}:

\begin{example} Let $k$ be an infinite field and let $R$ be the complete (2,1)-scroll, that is, $R=k[[x,y,z,u,v]]/I$ with $I$ the ideal generated by the 2x2 minors of  $\left(\begin{smallmatrix} x& y & u \\ y & z & v\end{smallmatrix}\right)$. Then, $R$ is a 3-dimensional CM normal domain of finite CM type \cite{Yoshino}*{ 16.12}. It is known  $\Gamma=\End_R(R\oplus \omega)$ is MCM over $R$, and $\Gamma$ is symmetric since it is an endomorphism ring over a normal domain \cite{IW1}*{Lemma 2.10}. But, Smith and Quarles have shown $\gldim(\Gamma)=4$ \cite{SQ} while $\dim R=3$. Thus $\Gamma$ is a symmetric $R$-order of finite global dimension but it not non-singular, thus $\Gamma$ does not provide a NCCR.
\label{scroll}
\end{example}

In another example, we see that $\End_R(R\oplus \omega)$ does yield a noncommutative crepant resolution. This is the only other known non-Gorenstein ring of dimension 3 with finite CM type \cite{Leuschke2}*{Example P.4}:

\begin{example} Let $R=k[[x^2,xy,xz,y^2,yz,z^2]]$, the second Veronese subring in three variables. Then, $R$ is known to have finite CM type \cite{Yoshino}*{16.10} with indecomposable  MCM modules $R$, $\omega\cong(x^2,xy,xz)$ and $M:=\syz_R(\omega)$. The ring $A:=\End_R(R\oplus \omega \oplus M)$ has global dimension $3$, but it is not MCM. Indeed, $A$ has depth 2, as both $\Hom_R(M,R)$ and $\Hom_R(M,M)$ have depth 2. 

It is easy enough to fix this example: $\Lambda=\End_R(R\oplus \omega)$ is in fact a noncommutative crepant resolution. This is because $R\oplus \omega\cong k[[x,y,z]]$ and so $\End_R(R\oplus \omega)$ is isomorphic to the twisted group ring $k[[x,y,z]]*\ZZ_2$. This is known to have global dimension 3  and be MCM over $R$ \cite{Yoshino}*{Ch. 10}. 
\label{2veron}
\end{example}

\begin{remark}The key point in Theorem \ref{vdb2} is the fact that over a Gorenstein ring of dimension $d$, a symmmetric order of finite global dimension is actually non-singular, and in fact even homologically homogeneous \cite{VDB1}*{Lemma 4.2}. Over a non-Gorenstein ring, these three conditions (symmetric, Maximal Cohen-Macaulay, and finite global dimension), do not guarantee non-singularity, as seen in Example \ref{scroll}. \label{rmk1} \end{remark}



It would be helpful to have an analog of Theorem \ref{Gor_equiv} to produce examples. In order to rescue some of the results from the prior case, we strengthen the hypotheses on $\Lambda=\End_R(M)$. Since the crepant condition (i.e., that $\Lambda$ is maximal Cohen-Macaulay) can be seen as a type of symmetry condition, one might hope to impose more stringent symmetry requirements on $\Lambda$. We recall the following definition:

\begin{definition}
An $R$-module is \emph{totally reflexive} if $M$ is reflexive and $\Ext_R^i(M,R)=\Ext_R^i(M^*,R)=0$ for all $i>0$, where $M^*=\Hom_R(M,R)$. 
\end{definition}

\begin{definition} A \emph{strong NC resolution} of a CM normal domain $R$ is $\Lambda=\End_R(M)$ for a reflexive $R$-module $M$ such that $\Lambda$ is totally reflexive over $R$ and of finite global dimension.\end{definition}

\begin{remark} This definition agrees with the original definition in the Gorenstein case since over a Gorenstein ring the totally reflexive modules are exactly the MCM modules. \end{remark}

We immediately see the trouble with this definition. 

\begin{prop} Let $(R,m,k)$ be a CM local ring, and $\Lambda$ a module-finite $R$-algebra such that $\Lambda^*=\Hom_R(\Lambda,R)$ has finite injective dimension as a left $\Lambda$-module. Additionally suppose $\Ext^i_R(\Lambda,R)=0$ for all $i>0$. Then $R$ is Gorenstein.  

\end{prop}

Before the proof, we need the following lemma: 

\begin{lem} \label{lem1} Let $R$ be a Cohen-Macaulay normal domain and $\Lambda$ an $R$-algebra such that $\Ext_R^n(\Lambda,R)=0$ for all $n>0$. For all $i>0$ and all left $\Lambda$-modules $B$ , we have that \[\Ext^i_R(B,R)\cong\Ext_\Lambda^i(B,\Hom_R(\Lambda,R)).\]
\end{lem} 
\begin{proof} This follows from the collapsing of the well known (see e.g., \cite{CE}*{Chapter XVI, Section 4}) spectral sequence \[\Ext_\Lambda^p(B,\Ext^q_R(\Lambda,C)) \Rightarrow_p \Ext^{p+q}_R(B,C).\qedhere\]
\end{proof}

And now to prove the Proposition.

\begin{proof}[Proof of Proposition]

\noindent We must only note that $\Lambda/m\Lambda$ is a finite-dimensional vector space over $k$, and we then have that $\Ext^i_R(k,R)$ is a summand of $\Ext^i_R(\Lambda/m\Lambda,R)$. Since $\Lambda^*$ has finite injective dimension,we have that $$\Ext^i_R(\Lambda/m\Lambda,\Lambda^*)=0$$ for $i$ sufficiently large. Then by Lemma \ref{lem1} we have that $\Ext^i_R(\Lambda/m\Lambda,R)$ and hence $\Ext^i_R(k,R)$ is zero for all $i$ sufficiently large. \end{proof}

\begin{cor} If $R$ is a CM local normal domain possessing a strong NC resolution $\Lambda$, then $R$ is Gorenstein.
\end{cor}

\section{Gorenstein Orders}

When $R$ is Gorenstein, one of the crucial components to Theorem \ref{Gor_equiv} is that a symmetric $R$-order $\Lambda$ of finite global dimension is homologically homogeneous and thus has many desirable homological properties.  We begin with a definition which gives some of the same properties.

\begin{definition} Let $R$ be a $d$-dimensional CM local normal domain with canonical module $\omega_R$. An $R$-order $\Lambda$ is \emph{a Gorenstein order} if $\omega_\Lambda:=\Hom_R(\Lambda, \omega_R)$ is a projective left $\Lambda$-module. 
\end{definition}

If we demand our order $\Lambda:=\End_R(M)$ be a Gorenstein order and of finite global dimension, Remark \ref{rmk1} implies we rescue Theorem \ref{vdb2}.

\begin{theorem}\cite{IW1}*{Proposition 2.17} Let $\Lambda$ be an order over a Cohen-Macaulay local ring $R$ with canonical module $\omega_R$. The following are equivalent:
\begin{itemize} \item[(i)] $\Lambda$ is non-singular.
\item[(ii)] $\Lambda$ is a Gorenstein order and of finite global dimension.
\end{itemize}
\label{IWGor}
\end{theorem}

\begin{remark} In both Examples \ref{scroll} and \ref{2veron} we examine the ring $\End_R(R\oplus \omega)$. In Example \ref{scroll} $\End_R(R\oplus \omega)$ is not  a Gorenstein order, but in Example \ref{2veron} $\End_R(R\oplus\omega)$ is non-singular, and hence is Gorenstein. In view of these examples, the following question is motivated.

\end{remark}

\begin{question} For a Cohen-Macaulay normal domain $R$ with canonical module $\omega$, when is \mbox{$\End_R(R\oplus \omega)$} a Gorenstein order of finite global dimension? \label{question} \end{question}

The following theorem will allow us to provide a partial answer by specializing to $I=\omega$. 


\begin{theorem} Suppose $R$ is a henselian local ring and $I$ is an indecomposable ideal which contains a nonzerodivisor. Then $\End_R(R\oplus I) \cong \Hom_R(\End_R(R \oplus I),I)$ as left $\End_R(R\oplus I)$-modules if and only if $I\cong I^*$ as $R$-modules. \label{main3}

\end{theorem}

\begin{proof} ($\Rightarrow$) Let $\Lambda=\End_R(R\oplus I)$.  We see, $$\Lambda=\begin{pmatrix} \Hom_R(R,R) & \Hom_R(I,R) \\  \Hom_R(R,I) &  \Hom_R(I,I) \end{pmatrix} \cong \begin{pmatrix} R & I^* \\ I & \End_R(I) \end{pmatrix}.$$ 

\vspace{.5cm}

The bimodule structure on $\Hom_R(\Lambda,I)$ is given by taking $\Hom_R(-,I)$ in each component and taking the transpose. Thus we have $$\Hom_R(\Lambda,I) = \begin{pmatrix} I & \End_R(I) \\  \Hom_R(I^*,I) &  I \end{pmatrix}.$$

Since $R$ and $I$ are indecomposable, when we decompose as left modules, we take the column vectors. Thus if $\Lambda \cong \Hom_R(\Lambda,I)$ one of the following must hold
 \begin{align*}
\begin{pmatrix} I  \\ \Hom_R(I^*,I) \end{pmatrix} \cong & \begin{pmatrix} R \\ I\end{pmatrix} \\  \begin{pmatrix}   I \\\Hom_R(I^*,I) \end{pmatrix} \cong & \begin{pmatrix} I^* \\ \End_R(I)\end{pmatrix}.
\end{align*}
Thus, either $I\cong R$ or $I\cong I^*$. In either case, the result holds.

\vspace{.5cm}

($\Leftarrow$) We wish to show that, $\Hom_R(\Lambda,I)$ is a free left $\Lambda$-module. We identify 
$$\Lambda=\begin{pmatrix} \Hom_R(R,R) & \Hom_R(I,R) \\  \Hom_R(R,I) &  \Hom_R(I,I) \end{pmatrix} \cong \begin{pmatrix} R & I^* \\ I & \End_R(I) \end{pmatrix}$$ as a ring, and $\Lambda =R \oplus I^* \oplus I \oplus \End_R(I)$ as an $R$-module. This allows us to identify $$\Hom_R(\Lambda,I) = \Hom_R(R,I) \oplus \Hom_R(I^*,I) \oplus \Hom_R(I,I) \oplus \Hom_R(\End_R(I),I)$$ where we see that the action of $\Lambda$ on $\Hom_R(\Lambda,I)$ is given by $(\lambda \cdot g) (\eta)=g(\eta \cdot\lambda)$ for $\lambda \in \Lambda$. 

We choose an isomorphism $\phi:I^* \to I$ and show that $f=\begin{bmatrix} 0 & \phi & 1 & 0 \end{bmatrix}$ is a basis for the left $\Lambda$-module $\Lambda^I$. Indeed, suppose we have a map $g\in \Hom_R(\Lambda, I)$, i.e., \[g=\begin{bmatrix} g_1 & g_2 & g_3 & g_4 \end{bmatrix} \in \Hom_R(R,I) \oplus \Hom_R(I^*,I) \oplus \Hom_R(I,I) \oplus \Hom_R(\End_R(I),I).\] We wish to show that there is a \[\lambda=\begin{pmatrix} \lambda_1 & \lambda_2 \\ \lambda_3 & \lambda_4\\ \end{pmatrix} \in \Lambda\] so that \begin{equation}g(\eta)=\lambda \cdot f(\eta)=f(\eta\cdot\lambda) \label{eq3} \end{equation} for all $\eta\in \Lambda$ and $f=\begin{bmatrix} 0 & \phi & 1 & 0 \end{bmatrix}$ . Set $$\eta= \begin{pmatrix} r & f \\ a & \sigma \end{pmatrix}$$ for $r\in R$, $f\in I^*$, $a\in I$ and $\sigma\in \End_R(I)$. Computing the left hand side of (\ref{eq3}) we get

$$g(\eta)=g_1(r)+g_2(f)+g_3(a)+g_4(\sigma).$$ The right hand side becomes 

\begin{align*}
f(\eta \cdot \lambda)=&f \begin{pmatrix} r\lambda_1+ f \lambda_3 & r\lambda_2+ f \lambda_4 \\ a\lambda_1+ \sigma \lambda_3 & a\lambda_2+ \sigma \lambda_4\\ \end{pmatrix} \\
=& \phi(r\lambda_2+ f \lambda_4) + a(\lambda_1)+ \sigma (\lambda_3) \\
=& \phi(r \lambda_2)+\phi(f \lambda_4)+  a(\lambda_1)+ \sigma (\lambda_3). \\
\end{align*} The only choice we have to make is that of $\lambda_1\in R, \lambda_2\in I^*, \lambda_3 \in I,$ and $\lambda_4\in \End_R(I)$. We note that since $I$ contains a nonzerodivisor, we have that $\End_R(I)$ is contained in the total quotient ring of $R$ and thus that every $R$-linear morphism is also $\End_R(I)$-linear. It follows at once that $\Hom_R(\End_R(I),I)\cong I$. Now, since $g_4 \in \Hom_R(\End_R(I),I)\cong I$ it is really just multiplication by $g_4(1) \in I$ and hence we have $g_4(\sigma)=\sigma g_4(1)$ and so we can simply choose $\lambda_3=g_4(1)\in I$. The same argument works for choosing $\lambda_1=g_3(1)\in \End_R(I)$ since $\End_R(I)$ is actually contained in the total quotient ring of $R$. Since $r\in R$ we have that $\phi(r \lambda_2)=\phi(\lambda_2)r$; but, $\phi$ is an isomorphism, so we can choose $\lambda_2=\phi^{-1}(g_1(1))$ so that $\phi(\lambda_2)=g(1)$ and as before $\phi(\lambda_2 r)=\phi(\lambda_2)r=g_1(1)r=g_1(r)$. Similarly we have that $\phi(f\lambda_4)=\lambda_4 \phi(f)$ since $\lambda_4 \in \End_R(I)$, but then we choose $\lambda_4=g_2\phi^{-1}\in \Hom_R(I,I)$. It follows that $\phi(f\lambda_4)=\lambda_4\phi(f)=g_2\phi^{-1}\phi(f)=g_2(f)$ and thus we have $$\phi(r \lambda_2)+\phi(f \lambda_4)+  a(\lambda_1)+ \sigma (\lambda_3) = g_1(r)+g_2(f)+g_3(a)+g_4(\sigma).$$ This concludes the proof.\end{proof}

\begin{cor}  Suppose $R$ is a CM henselian generically Gorenstein ring with canonical module $\omega$. Then  $\End_R(R \oplus \omega)$ is a Gorenstein $R$-order if and only if $\omega \cong \omega^*$. In particular if $R$ is a CM local domain, then this is further equivalent to $[\omega]$ having order 2 in the divisor class group of $R$.

\label{MainCor}
\end{cor}

\begin{proof} $(\Leftarrow):$ this is a direct application of Theorem \ref{main3}.
 
$(\Rightarrow):$ We must address the fact that $\Lambda$ being a Gorenstein order may not imply that $\Hom_R(\Lambda,\omega)\cong\Lambda$ but only that $\Hom_R(\Lambda,\omega)$ is a summand of a finite sum of copies of $\Lambda$. We note that as $R$-modules \begin{align*}
	\Lambda\cong & \ R^2 \oplus \omega \oplus \omega^* \\
    \Hom_R(\Lambda,\omega)\cong & \ \omega^2\oplus R \oplus \Hom_R(\omega^*,\omega).
\end{align*}

Since the functor $\Hom_R(R\oplus\omega,-):\add(R\oplus\omega) \to \text{proj}\Lambda$ is an equivalence, we have that the indecomposable projectives $\Lambda$-modules are $P_1=\Hom_R(R\oplus\omega,R)$ and $P_2=\Hom_R(R\oplus\omega,\omega)$. As $R$-modules, we have 
\begin{align*}
P_1\cong & \ R\oplus \omega^* \\
P_2\cong & \ R\oplus \omega.\end{align*} By considering ranks, we see $\Hom_R(\Lambda,\omega)\cong P_i\oplus P_j$ for $i,j\in\{1,2\}$. In the case that $\Hom_\Lambda(\Lambda,\omega)$ is $P_1^2$ or $P_1\oplus P_2$ it is clear that either $R$ is Gorenstein or  $\omega^*\cong \omega$ since $R$, $\omega$, and $\omega^*$ are all rank one $R$-modules; in either case, the result holds. Thus, we must only deal with the case \[\Hom_\Lambda(\Lambda,\omega)\cong P_2^2\] so that we deduce $\Hom_R(\omega^*,\omega)\cong R$. But, as $R$ is CM, we have that $R$ satisfies Serre's condition $(S_1)$ and we know $\omega^*$ satisfies $(S_2)$ as it is the dual of a finitely generated $R$-module. Thus, we know $\omega^*$ is $\omega$-reflexive by \cite{Hartshorne}*{Lemma 1.5}, and we see \[\omega^*\cong \Hom_R(\Hom_R(\omega^*, \omega),\omega) \cong \Hom_R(R,\omega) \cong \omega. \qedhere \] \end{proof}

\subsection{Examples}

Since we now have a criterion on the canonical module, we will consider some invariant subrings under actions by cyclic groups, where the order of $\omega_R$ is easily computed. We start with a theorem of Weston \cite{Weston}. Throughout we assume $k$ is algebraically closed of characteristic 0.  

\begin{Hypo} Let $S=k[x_1, \dots, x_n]$ and $G$ a finite  subgroup of $GL_n(k)$ with generators $g_1,\dots,g_t$ which acts linearly on the variables. For each $j=1,\dots,t$ let $\zeta_j$ be a primitive $|g_j|^{th}$ root of unity in $k$. Then for each $j$ there exists a basis for $kx_1 \oplus \dots \oplus kx_n$ so that $$g_j=\begin{bmatrix} \zeta_j^{a_{1j}} & 0 & \dots & 0 \\ 0 & \zeta_j^{a_{2j}} & \dots & 0 \\
0 & \dots & \zeta_j^{a_{3j}} &  0 \\ \vdots & & \ddots & \vdots\\ 0  & \dots &  0 & \zeta_j^{a_{nj}} \\ \end{bmatrix}$$ for integers $a_{ij}$ with $1\leq a_{ij} < |g_j|$ for $j=1,\dots,t$. Set $d_{ij}=\gcd(a_{1j}\dots,\widehat{a_{ij}},\dots a_{nj},|g_j|)$ and $m_j$ the least integer so that $m_j\sum_{i=1}^n d_{ij}a_{ij}=0$ mod$|g_j|$. Let $R=S^G$. When $G$ is cyclic (i.e., t=1), we will suppress the use of $j$, as it is not needed. 
\label{hypoWeston}

\end{Hypo}



\begin{theorem}\cite{Weston}*{Theorem 2.2} With Hypotheses \ref{hypoWeston} the class $[\omega]$ in $\text{Cl}(R)$ of the canonical $R$ module $\omega$ has order $m:=\text{l.c.m}(m_1,\dots,m_t)$

\label{West1}
\end{theorem}

\begin{cor} Let $S=k[x_1,\dots,x_n]$ (or, $k[[x_1,\dots,x_n]]$ ) and $G\subset \text{GL}(n,k)$ be a finite subgroup acting linearly on the variables and set $R=S^G$. If $\End_R(R\oplus \omega)$ is a Gorenstein $R$-order, then $G$ is of even order. 
\end{cor}

\begin{proof} 

First we note that by \cite{Weston2}, it suffices to treat the polynomial ring case. We adopt the notation of Hypotheses \ref{hypoWeston}. By Theorem \ref{West1} the order of $\omega$ is $\text{lcm}(m_1, \dots , m_t)$. Now,  Corollary \ref{MainCor} says that if $\End_R(R\oplus \omega)$ is a Gorenstein algebra, then we have $|[\omega]|=2$, since $R$ is normal. This means at least one $m_i=2$, call it $m_1$. Then we have that $m_1\sum_{j=1}^n a_{1j}d_{1j}=l|g_1|$. But as 2 is prime, it must divide $l$ or $|g_1|$. It cannot divide $l$ as then a smaller integer would be chosen instead of $m_1$. Thus it must be that $|g_1|$ is even, and hence $G$ must be of even order.\end{proof}

\begin{remark} Note that the converse to this is not true, since it is possible to have $|G|$ even, but the order of the canonical module not be 2. For example, let $R=k[[x_1,x_2, x_3, x_4]]^{(2)}=k[[x_ix_j]]_{1\leq i\leq j\leq 4}$, the second Veronese subring of $k[[x_1,x_2,x_3,x_4]]$.  Here $G$ is cyclic of order 2, $G\subset \text{SL}_2(k)$, and $G$ contains no psuedo-reflections, hence $R$ is Gorenstein \cite{Watanabe1}*{Theorem 1}. It follows that the order of $[\omega]$ is 1. \end{remark}
 
This gives us the ability to produce ample examples of Gorenstein orders over the power series ring in $n$ variables.

\section{Steady NCCRs and global dimension}

In this section we will suppose we are working within the conditions of \ref{hypoWeston}. We start with the following definitions, from \cite{Iyama_steady}:

\begin{definition} Let $R$ be a $d$-dimensional CM local normal domain with canonical module $\omega_R$.
\mbox{}
\begin{itemize}
\item A module $M$ is \emph{steady} if it is a generator and $\End_R(M)\in \text{add}_R M$. 
\item If $M$ is steady and $\End_R(M)$ is a noncommutative crepant resolution of $R$ then we say $\End_R(M)$ is a \emph{steady NCCR}. 
\item If $M$ is a direct sum of reflexive modules of rank one, then we call $M$ \emph{splitting}. 
\item We say $\End_R(M)$ is a \emph{splitting NCCR} if it is an NCCR and $M$ is splitting.
\item If $M=M_1\oplus \cdots \oplus M_n$ is a decomposition of $M$ into indecomposables, we say $M$ is \emph{basic} if the $M_i$ are mutually nonisomorphic.
\end{itemize}

\end{definition}

\begin{remark} Let $R$ be a $d$-dimensional CM local normal domain with canonical module $\omega_R$. We see that if the conditions of Corollary \ref{MainCor} are satisfied, then $\End_R(R\oplus \omega)\cong R\oplus R\oplus \omega \oplus \omega \in \text{add}_R(R \oplus \omega)$ so that $R \oplus \omega$ is a steady splitting module. If $R$ is not Gorenstein, then $R\oplus \omega$ is basic.
\label{splittingremark}
\end{remark}

\begin{theorem} Let $R=S^G$ be a subring of $S=k[[x_1,\dots,x_n]]$ for $k$ an algebraically closed field of characteristic zero with $G$ a small (i.e., contains no pseudo-reflections) abelian subgroup of $\text{GL}_n(k)$ and such that $[\omega]$ has order 2 in $\text{Cl}(R)$.  Then $\End_R(R\oplus\omega)$ has finite global dimension if and only if $R$ is isomorphic to a ring of the form $T[[x_{j+1},\dots,x_{n}]]$ for $j$ odd and $3 \leq j \leq n$. where $T=k[[x_1,\dots, x_j]]^{(2)}$. 
\end{theorem}

Before the proof, we need the following result of Iyama and Nakajima:

\begin{lemma}\cite{Iyama_steady}
Let $R$ be a $d$-dimensional CM local normal domain.  Then the following are equivalent:

\begin{itemize}
\item $R$ is a quotient singularity associated with a finite abelian group $G\subset Gl_d(k)$ (i.e., $R=S^G$ where $S=k[[x_1,\dots,x_d]]$.) 
\item $R$ has a unique basic module giving a splitting NCCR.
\item $R$ has a steady splitting NCCR. 
\end{itemize}

In this case, $S$ is the unique basic splitting module giving an NCCR.
\label{I_steady}
\end{lemma}

\begin{proof}[Proof of Theorem 5.3] ($\Rightarrow$): Suppose that $\End_R(R\oplus \omega)$ has finite global dimension. Since $|[\omega]|=2$ we know that $\End_R(R\oplus \omega)$ is a Gorenstein $R$-order by Corollary \ref{MainCor}. Thus, by \ref{IWGor}, $\End_R(R \oplus \omega)$ is non-singular. Then, $\End_R(R\oplus \omega)$ is a noncommutative crepant resolution of $R$ since it is certainly an MCM $R$-module.  This NCCR is steady and splitting and $R\oplus\omega$ is basic by Remark \ref{splittingremark}. Thus  $R\oplus \omega \cong S$ by Lemma \ref{I_steady}. It follows from Galois theory that $|G|=\rank_RS=2$. We see immediately that $n\geq 3$ since otherwise $R$ would be Gorenstein and $|[\omega]|=1$.  Then $G=\langle \sigma \rangle$ where $\sigma^2=1$. Since $G$ is small, the order of the cyclic subgroup ${\det(G)=\{ \det(g) \ | g \in G\}} \leq k^*$ is the order of the canonical module of $R$ in the divisor class group (for reference see \cite{BH}*{Theorem 6.4.9} and the comments after Theorem 1 in \cite{Watanabe2}). Since $|[\omega]|=2$ we must have that $\det(\sigma)=-1$. Further, as $|\sigma|=2$, the minimal polynomial of $\sigma$ is $(t-1)(t+1)$ and so $\sigma$ is diagonalizable with eigenvalues $\lambda=\pm 1$. Thus there is a basis for $V=kx_1\oplus kx_2 \oplus \dots \oplus kx_n$ where 

	\[\sigma=\begin{pmatrix}
		1 & 0 & 0 & \dots & 0 \\
		0 & 1 & 0 & \dots & 0 \\
%
        0 & 0 & -1 & 0 & \dots \\
        \vdots & & \ddots & & \vdots \\
        0 & 0 & 0 & 0 & -1 \\
\end{pmatrix}.\] 

\vspace{.2cm}

Note that the number of negative entries is exactly the quantity $j$ from the theorem. We wish to show $j$ is at least 3 and odd.  Since $\det(\sigma)=-1,$ $j$ must be odd, so we must only deal with the case $j=1$. However, $j=1$ would mean that $\sigma$ is a pseudo-reflection, hence $G$ would not be small. Therefore, $j\geq 3$. It follows at once that $R$ is of the indicated form. 


($\Leftarrow$): Suppose $R$ is of the indicated form. As above, we then have $d_i=1$ for $i=1,\dots,n$. Then by \cite{Weston}*{Example 2.3} $\omega \cong x_1x_2\dots x_nS\cap R \cong (x_j,\dots, x_n)$ and hence $R\oplus \omega \cong k[[x_1,\dots,x_n]]$ as $R$-modules and thus, $\End_R(R \oplus \omega) \cong \End_R(S)$ which is known to have finite global dimension, see \cite{Iyama_steady}*{Example 2.3}.

\end{proof}

\begin{remark} It should be noted that the condition $|[\omega]|=2$ (in particular, that $R$ is not Gorenstein) is needed. If we do not require this, the theorem is false. \end{remark}

\begin{example} Let $R=k[[x,y]]^{\ZZ_2}$ where the group acts via $x\mapsto y$ and $y\mapsto x$. Then $R\cong k[[xy,x+y]]$ and hence is a regular local ring. Thus $\End_R(R\oplus \omega)=\End_R(R^2)$ has finite global dimension as it is Morita equivalent to $R$. Similar examples exist for larger $n$. 

\end{example}

\section{Acknowldgements}

The author would like to thank his adviser, Graham  J. Leuschke, for much insight and guidance, and the NSA for support via grant \# H98230-13-1-024. He would also like to thank the referee for many helpful suggestions and refinements.


\linespread{1}\selectfont



\newpage
\bibliographystyle{plain}
\bibliography{biblio}

\begin{bibdiv}
\begin{biblist}

\bib{BH}{book}{
      author={Bruns, Winfried},
      author={Herzog, J{\"u}rgen},
       title={Cohen-{M}acaulay rings},
      series={Cambridge Studies in Advanced Mathematics},
   publisher={Cambridge University Press, Cambridge},
        date={1993},
      volume={39},
        ISBN={0-521-41068-1},
      review={\MR{1251956}},
}

\bib{BLV1}{article}{
      author={Buchweitz, Ragnar-Olaf},
      author={Leuschke, Graham~J.},
      author={Van~den Bergh, Michel},
       title={Non-commutative desingularization of determinantal varieties
  {I}},
        date={2010},
        ISSN={0020-9910},
     journal={Invent. Math.},
      volume={182},
      number={1},
       pages={47\ndash 115},
         url={http://dx.doi.org/10.1007/s00222-010-0258-7},
      review={\MR{2672281}},
}

\bib{CE}{book}{
      author={Cartan, Henri},
      author={Eilenberg, Samuel},
       title={Homological algebra},
      series={Princeton Landmarks in Mathematics},
   publisher={Princeton University Press, Princeton, NJ},
        date={1999},
        ISBN={0-691-04991-2},
        note={With an appendix by David A. Buchsbaum, Reprint of the 1956
  original},
      review={\MR{1731415}},
}

\bib{DaoExample}{article}{
      author={Dao, Hailong},
       title={Remarks on non-commutative crepant resolutions of complete
  intersections},
        date={2010},
        ISSN={0001-8708},
     journal={Adv. Math.},
      volume={224},
      number={3},
       pages={1021\ndash 1030},
         url={http://dx.doi.org/10.1016/j.aim.2009.12.016},
      review={\MR{2628801}},
}

\bib{DIF}{article}{
      author={Dao, Hailong},
      author={Faber, Eleonore},
      author={Ingalls, Colin},
       title={Noncommutative (crepant) desingularizations and the global
  spectrum of commutative rings},
        date={2015},
        ISSN={1386-923X},
     journal={Algebr. Represent. Theory},
      volume={18},
      number={3},
       pages={633\ndash 664},
         url={http://dx.doi.org/10.1007/s10468-014-9510-y},
      review={\MR{3357942}},
}

\bib{DITV}{article}{
      author={Dao, Hailong},
      author={Iyama, Osamu},
      author={Takahashi, Ryo},
      author={Vial, Charles},
       title={Non-commutative resolutions and {G}rothendieck groups},
        date={2015},
        ISSN={1661-6952},
     journal={J. Noncommut. Geom.},
      volume={9},
      number={1},
       pages={21\ndash 34},
         url={http://dx.doi.org/10.4171/JNCG/186},
      review={\MR{3337953}},
}

\bib{Hartshorne}{article}{
      author={Hartshorne, Robin},
       title={Generalized divisors and biliaison},
        date={2007},
        ISSN={0019-2082},
     journal={Illinois J. Math.},
      volume={51},
      number={1},
       pages={83\ndash 98 (electronic)},
         url={http://projecteuclid.org/euclid.ijm/1258735326},
      review={\MR{2346188}},
}

\bib{Iyama_steady}{article}{
      author={Iyama, Osamu},
      author={Nakajima, Yusuke},
       title={On steady non-commutative crepant resolutions},
        date={2015},
     journal={arXiv preprint arXiv:1509.09031},
}

\bib{IW1}{article}{
      author={Iyama, Osamu},
      author={Wemyss, Michael},
       title={Maximal modifications and {A}uslander-{R}eiten duality for
  non-isolated singularities},
        date={2014},
        ISSN={0020-9910},
     journal={Invent. Math.},
      volume={197},
      number={3},
       pages={521\ndash 586},
         url={http://dx.doi.org/10.1007/s00222-013-0491-y},
      review={\MR{3251829}},
}

\bib{Leuschke2}{incollection}{
      author={Leuschke, Graham~J.},
       title={Non-commutative crepant resolutions: scenes from categorical
  geometry},
        date={2012},
   booktitle={Progress in commutative algebra 1},
   publisher={de Gruyter, Berlin},
       pages={293\ndash 361},
      review={\MR{2932589}},
}

\bib{SQ}{article}{
      author={Smith, S.~Paul},
      author={Quarles, Christopher},
       title={On an axample of leuschke},
        date={2005},
     journal={Private communication},
}

\bib{VDBSTAFF}{article}{
      author={Stafford, J.~T.},
      author={Van~den Bergh, M.},
       title={Noncommutative resolutions and rational singularities},
        date={2008},
        ISSN={0026-2285},
     journal={Michigan Math. J.},
      volume={57},
       pages={659\ndash 674},
         url={http://dx.doi.org/10.1307/mmj/1220879430},
        note={Special volume in honor of Melvin Hochster},
      review={\MR{2492474}},
}

\bib{VDB1}{incollection}{
      author={van~den Bergh, Michel},
       title={Non-commutative crepant resolutions},
        date={2004},
   booktitle={The legacy of {N}iels {H}enrik {A}bel},
   publisher={Springer, Berlin},
       pages={749\ndash 770},
      review={\MR{2077594}},
}

\bib{Watanabe1}{article}{
      author={Watanabe, Keiichi},
       title={Certain invariant subrings are {G}orenstein. {I}, {II}},
        date={1974},
        ISSN={0030-6126},
     journal={Osaka J. Math.},
      volume={11},
       pages={1\ndash 8; ibid. 11 (1974)},
      review={\MR{0354646}},
}

\bib{Watanabe2}{article}{
      author={Watanabe, Keiichi},
       title={Certain invariant subrings are {G}orenstein. {I}, {II}},
        date={1974},
        ISSN={0030-6126},
     journal={Osaka J. Math.},
      volume={11},
       pages={379\ndash 388},
      review={\MR{0354646}},
}

\bib{Weston2}{article}{
      author={Weston, Dana},
       title={Divisorial properties of the canonical module for invariant
  subrings},
        date={1991},
        ISSN={0092-7872},
     journal={Comm. Algebra},
      volume={19},
      number={9},
       pages={2641\ndash 2666},
         url={http://dx.doi.org/10.1080/00927879108824285},
      review={\MR{1125195}},
}

\bib{Weston}{article}{
      author={Weston, Dana},
       title={Stability of the divisor class group upon completion},
        date={2007},
        ISSN={0019-2082},
     journal={Illinois J. Math.},
      volume={51},
      number={1},
       pages={313\ndash 323},
         url={http://projecteuclid.org/euclid.ijm/1258735338},
      review={\MR{2346200}},
}

\bib{Yoshino}{book}{
      author={Yoshino, Yuji},
       title={Cohen-{M}acaulay modules over {C}ohen-{M}acaulay rings},
      series={London Mathematical Society Lecture Note Series},
   publisher={Cambridge University Press, Cambridge},
        date={1990},
      volume={146},
        ISBN={0-521-35694-6},
         url={http://dx.doi.org/10.1017/CBO9780511600685},
      review={\MR{1079937}},
}

\end{biblist}
\end{bibdiv}

\end{document}